\documentclass[smallextended,referee,envcountsect,,numbers]{svjour3}
\smartqed
\usepackage[numbers]{natbib}
\usepackage{graphicx}%
\usepackage{multirow}%
\usepackage{amsmath,amssymb,amsfonts}%
\usepackage{mathrsfs}%
\usepackage[title]{appendix}%
\usepackage{xcolor}%
\usepackage{textcomp}%
\usepackage{manyfoot}%
\usepackage{booktabs}%
\usepackage{algorithm}%

\usepackage{algpseudocode}%

\usepackage{listings}%
\usepackage{todonotes}

\usepackage{amssymb}
\usepackage{setspace}
\usepackage{amsmath}
\usepackage{moreverb}
\usepackage{caption}

\journalname{}

\usepackage[hidelinks]{hyperref}
\begin{document}
\title{Sparse Polynomial Regression under Anomalous Data}

\author{Roozbeh Abolpour         \and
        Mohammad Reza Hesamzadeh \and 
        Maryam Dehghani
}

\institute{R. Abolpour \at
              Energy Information Networks and Systems, Technical University of Darmstadt, Darmstadt, Germany \\
              \email{roozeh.abolpour@eins.tu-darmtadt.de}                \and
           M. R. Hesamzadeh\at
           School of Electrical Engineering and Computer Science, KTH Royal Institute of Technology, Stockholm, Sweden
           \email{mrhesa@kth.se}
           \and
           M. Dehghani\at
           Department of Power and Control Engineering, School of Electrical and Computer Engineering, Shiraz University, Shiraz, Iran
           \email{mdehghani@shirazu.ac.ir}
}

\date{Received: date / Accepted: date}

\maketitle

\abstract{This paper starts with the general form of the polynomial regression model. We reformulate the Sparse Polynomial Regression Model (SPRM) with anomalous data filtering as Mixed-Integer Linear Program (MILP). This MILP is then converted to a non-convex Quadratically Constrained Quadratic Program (QCQP). Through a proposed mapping, the derived QCQP is reformulated as a Fractional Program (FP). We theoretically show that the reformulated FP has better computational properties than the original QCQP. We then suggest a conic-relaxation-based algorithm to solve the proposed FP. A Two-Step Convex Relaxation and Recovery (TS-CRR) algorithm is proposed for sparse polynomial regression with anomalous data filtering. Through a series of comprehensive computational experiments (using two different datasets), we have compared the results of our proposed TS-CRR algorithm with the results from several regression and artificial intelligent models. The numerical results show the promising performance of our proposed TS-CRR algorithm as compared to those studied benchmark models.} 

\keywords{Sparse Polynomial Regression Model (SPRM);  Fractional mapping; Mixed Integer Linear Program (MILP); Quadratically Constrained Quadratic Program (QCQP); Conic relaxation; Anomalous data.}  

\maketitle

\section{Introduction}
Regression models are fundamental tools in statistics and data science which are widely used in disciplines such as data engineering \cite{bertsimas2025slowly}, prediction \cite{bertsimas2020novel}, times series classification \cite{guijo2024unsupervised}, biomedical engineering \cite{hashemi2021efficient}, economics \cite{karakatsani2008forecasting}, and finance \cite{liu2021leveraging}. These models help uncover relationships between variables, support predictive analytics, and inform decision-making by analyzing historical data to understand how changes in independent variables affect a dependent variable \cite{james2021introduction}. Among various regression models, the polynomial regression is particularly valuable for capturing nonlinear patterns. They model the relationship between dependent and independent variables using polynomial functions with unknown parameters \cite{bridgman2025novel}. Authors in  \cite{bertsimas2020sparseS} discuss the statistical properties of a polynomial regression model. The hierarchical variable selection for polynomial regression models is proposed in \cite{bertsimas2020sparse}. The hierarchical polynomial regression model is applied in various domains, including predicting the global transmission of COVID-19 \cite{ekum2020application}, healthcare \cite{west2020necessity}, and marketing \cite{niu2025role}.

However, two key challenges arise when regression models are applied to real-world and high-dimensional data: (1) the presence of anomalous data that can distort model estimates \cite{filzmoser2021robust}, and (2) the rapid growth in polynomial terms, which can lead to overfitting and computational complexity \cite{bottmer2022sparse}. These challenges need to be properly addressed to ensure a practical and robust polynomial regression model. 

The first challenge is addressed in the literature by proposing robust polynomial regression approaches. Authors in \cite{arora2024robust} propose a polynomial regression model for cases where significant anomalous data exist. Their robust regression models have provable guarantees on estimation accuracy and computational efficiency. A robust local polynomial regression with similarity kernels is proposed in \cite{shulman2025robust} to mitigate the effect of anomalous data. It uses similarity kernels thereby improving robustness and estimation accuracy. Reference \cite{kane2017robust} introduces a theoretical advance in robust polynomial regression. It presents a polynomial-time algorithm for univariate polynomial regression that is robust to nearly 50\% adversarial outliers. Complementing this, \cite{zhang2025wind} using wind-power data, demonstrates how robust polynomial regression techniques by removing anomalous data can significantly improve the prediction performance. 

The second challenge is mainly tackled in the literature by proposing sparsity-based techniques.  Reference \cite{hastie2015statistical} provides the foundational theory and methods for sparse regression focusing on the Lasso (Least Absolute Shrinkage and Selection Operator) and its variations. Building on this foundation, a mixed-integer programming formulation for sparse polynomial regression is discussed in \cite{bertsimas2020sparse}. The problem is then cast as a bilevel minimization model and solved using a cutting-plane algorithm. In \cite{wang2022exploiting} sparsity-exploiting techniques for complex polynomial optimization is studied. The proposed approach in \cite{wang2022exploiting} can be potentially used for managing the computational complexities of polynomial regression models. In \cite{chen2023solution}, sparsity in multivariate regression is studied. It provides valuable theoretical support for selecting sparse modeling approaches in high-dimensional regression. A matching-pursuit–based sparse regression algorithm that is robust to adversarial corruptions in both predictors and responses is developed and suggested in \cite{chen2013robust}. 

Beyond classical regression context, \cite{sinha2024generalized} explores generalized sparse regression codes for communication systems showing its practical benefits. Furthermore, \cite{adcock2024learning} compares sparse polynomial approximations to deep neural networks showing that former can achieve competitive performance while offering the interpretability and computational advantage. 

With this background, the current paper contributes to the relevant literature as follows: 

\begin{itemize}
	\item The Sparse Polynomial Regression Model (SPRM) with anomalous data filtering is first formulated as a MILP and then as a nonconvex QCQP. Through an introduced fractional mapping, the nonconvex QCQP is reformulated as a Fractional Program (FP). 
	\item We theoretically prove that the reformulated FP has better computational properties than the original nonconvex QCQP. 
	\item We propose a conic-relaxation-based model to solve our proposed FP. Furthermore, a Two-Step Convex Relaxation and Recovery (TS-CRR) algorithm is proposed for sparse polynomial regression under anomalous data. 
	\item Through a set of comprehensive numerical experiments, we compare the results of our TS-CRR algorithm with those from several benchmark models. The comparative results show the promising performance of our proposed TS-CRR algorithm. 
\end{itemize}

This paper is organized as follows. Section \ref{sec:sprm} presents the problem formulation. Section \ref{sec:fp} derives the proposed FP reformulation. The proposed TS-CRR algorithm is discussed in Section \ref{sec:salg}. Two applications of our proposed regression approach under anomalous data are discussed in Section \ref{sec:app}. Finally, Section \ref{sec:cnc} concludes the paper.

\vspace{6cm}

\section{Problem Formulation} \label{sec:sprm}
\noindent Suppose ${\left\{x^{\left(k\right)}\right\}}^N_{k=1}$ and ${\left\{y^{\left(k\right)}\right\}}^N_{k=1}$ are the given input and output data, respectively. The objective is to determine the following polynomial regression model
\begin{equation} \label{eq1} 
y^{\left(k\right)}=\sum_{\alpha \in {\mathrm{\Gamma }}_d}{c_{\alpha }{x^{\left(k\right)}}^{\alpha }}+\epsilon^{\left(k\right)},\quad \forall k\in \left\{1,\dots ,N\right\}, 
\end{equation} 
where $d$ is the polynomial degree, ${\mathrm{\Gamma}}_d=\left\{\alpha\in {\mathbb{N}}^n\right|\ \sum^n_{i=1}{{\alpha }_i}\le d\}$ with $m_d$ distinct vectors (i.e. $m_d=|{\mathrm{\Gamma}}_d|$), $\epsilon^k$ is the error term, ${x^{\left(k\right)}}^{\alpha }=\prod^n_{i=1}{{x^{\left(k\right)}_i}^{{\alpha }_i}\ }$, and ${\left\{c_{\alpha }\right\}}_{\alpha \in {\mathrm{\Gamma }}_d}$ is the set of polynomial coefficients. The following optimization problem determines these polynomial coefficients.  

 \begin{subequations} 	\label{eq2}
	\begin{align}		
    \min_{\Omega_1}&\quad \gamma \\ \displaybreak[0] 
		s.t.&\quad \left|y^{\left(k\right)}-\sum_{\alpha \in {\mathrm{\Gamma }}_d}{c_{\alpha }{x^{\left(k\right)}}^{\alpha }}\right| \le \gamma, \quad \forall k\in \left\{1,\dots ,N\right\},
	\end{align}
\end{subequations}
where $\Omega_1=\{ c \in{\mathbb{R}}^{m_d},\gamma \in \mathbb{R}_+ \}$ and $\gamma$ is the regression error. 

Conventionally, a subset of the training data $\{(x_k,y_k)\}_{k=1}^{N}$ is considered anomalous if its members deviate from the underlying mathematical model of the system. Hence, detecting the anomalous data requires a mathematical model, which may be unavailable in some real-world applications. To address this limitation, an alternative model-independent definition is proposed below.
\begin{definition} \label{def:ad}
A subset $\bar{J}\subset \{1,...,N\}$ is defined to contain the indices of anomalous data if it satisfies the following condition.
\begin{equation} \label{eq3}
    \bar{J}=\underset{J \subset \{1,...,N\},|J|=|\bar{J}|}{\text{argmin}} \left(\max_{k\in \{1,...,N\} \setminus J} {\left|y^{\left(k\right)}-f(x^{(k)})\right|}\right),
\end{equation}
where $f(.)$ is the regression model.\qed
\end{definition}
This means, among all subsets $J$ of a given size (same as 
$\bar{J}$), we select the one whose removal from the dataset leads to the lowest possible maximum training error. This definition can be directly integrated into the training procedure, enabling the identification and elimination of anomalous data during training. 

To integrate both the anomalous-data definition \ref{def:ad} and the sparsity property into the polynomial regression optimization model \eqref{eq2}, binary decision variables are introduced. Let set ${\left\{s_{\alpha }\right\}}_{\alpha \in {\mathrm{\Gamma }}_d}$ include binary variables indicating whether each monomial term ${\left\{{x^{}}^{\alpha }\right\}}_{\alpha \in {\mathrm{\Gamma }}_d}$ is included in the regression model. In addition, let set ${\left\{b_k\right\}}_{k=1}^{N}$ be binary variables indicating the inclusion of the $k^{th}$ pair $(x^{(k)},y^{(k)})$ in training, where $b_k=0$ means the $k^{th}$ pair is considered anomalous and excluded. The sparse polynomial regression model with anomalous data elimination is formulated as the following MILP. 
\begin{subequations} 	\label{eq4}
	\begin{align}
        \underset{\Omega_2}{\min} & \quad \gamma \\ \displaybreak[0] 
		\text{s.t.} & \quad-M(1-b_k)+\left|y^{\left(k\right)}-\sum_{\alpha \in {\mathrm{\Gamma }}_d}{c_{\alpha }{x^{\left(k\right)}}^{\alpha }}\right|\le \gamma, \quad\forall k\in \left\{1,\dots ,N\right\}, \label{eq4_b} \\ \displaybreak[0] 
		&\quad -Ms_{\alpha }\le c_{\alpha }\le Ms_{\alpha }~\forall \alpha \in {\mathrm{\Gamma}}_d, \label{eq4_c} \\ \displaybreak[0]
		&\quad \sum_{k=1}^Nb_k=l_b ,\label{eq4_d} \\ \displaybreak[0]
        &\quad \sum_{\alpha \in \Gamma_d}s_\alpha=l_m,\label{eq4_e}
	\end{align}
\end{subequations}
where $\Omega_2=\left\{ c \in {\mathbb{R}}^{m_d},s\in {\left\{0,1\right\}}^{m_d},\\b \in \{0,1\}^N,\gamma \in \mathbb{R}_+  \right\}$, $M$ is a given large positive constant, and $l_b$ and $l_m$ are positive integer numbers. In MILP \eqref{eq4}, the training error $f(x)=\sum_{\alpha \in \Gamma_d}{s_\alpha c_\alpha x^\alpha}$ is minimized. In (\ref{eq4_b}), binary variable $b_k$ determines whether the pair $(x^{(k)},y^{(k)})$ contributes to the training error $\gamma$. Similarly, binary variable $s_\alpha$ in (\ref{eq4_c}) determines the inclusion of monomial term $x^\alpha$ in the training. Based on constraint (\ref{eq4_d}), the number of anomalous data is given by $N-l_b$ and due to constraint (\ref{eq4_e}), the final polynomial regression model will be sparse, with the number of monomials fixed at $l_m$.

\section{Fractional Programming (FP) Reformulation} \label{sec:fp}
\subsection{Reformulation Motivation}
The MILP formulation \eqref{eq4} effectively models sparsity and anomalous data filtering through introduced binary variables. However, solving MILPs at scale remains computational intensive, especially when the number of monomial terms and data points grow. To address this limitation, we propose a mathematical transformation through which the MILP is reformulated to a continuous-domain Fractional Program (FP) with improved computational properties. 

\subsection{Mathematical Transformation}
\noindent We first reformulate the MILP \eqref{eq4} into the following nonconvex  QCQP:
\begin{subequations} 	\label{eq5}
	\begin{align}
		\underset{\Omega_3}{\min}&\quad \gamma\\ \displaybreak[0]
        \text{s.t.}&\quad -M(1-b_k)+\left|y^{\left(k\right)}-\sum_{\alpha \in {\mathrm{\Gamma }}_d}{c_{\alpha }{x^{\left(k\right)}}^{\alpha }}\right|\le \gamma,\quad \forall k\in \left\{1,\dots ,N\right\}, \\ \displaybreak[0] 
		&\quad \ -Ms_{\alpha }\le c_{\alpha }\le Ms_{\alpha }~\forall \alpha \in {\mathrm{\Gamma }}_d, \\ \displaybreak[0]
		&\quad \sum_{k=1}^N{b_k}=l_b ,\\ \displaybreak[0]
        &\quad \sum_{\alpha \in \Gamma_d}s_\alpha=l_m ,\\ \displaybreak[0]
		&\quad \sum_{\alpha \in {\mathrm{\Gamma }}_d}{s_{\alpha }}+\sum_{k=1}^Nb_k\le {\sum_{\alpha \in {\mathrm{\Gamma }}_d}{s_{\alpha }^2}}+\sum_{k=1}^Nb_k^2, \label{eq5_f} 
	\end{align}
\end{subequations}
where $\Omega_3=\left\{c \in {\mathbb{R}}^{m_d},s\in {\left[0,1\right]}^{m_d},b \in [0,1]^N,\gamma \in \mathbb{R}_+ \right\}$.

\noindent It can be readily shown that the nonconvex QCQP \eqref{eq5} is equivalent to the MILP \eqref{eq4}, since constraint (\ref{eq5_f}), together with the bounds $s\in {\left[0,1\right]}^{m_d}$, and $b \in [0,1]^N$ enforce vectors $s$ and $b$ to take binary values; that is, $s\in \{0,1\}^{m_d}$ and $b\in \{0,1\}^{N}$. Now, consider the following mathematical mapping in which $\rho$ is a sufficiently large positive constant. 
\begin{align} \label{eq6} 
\left(\hat{c},\hat{s},\hat{b},\hat{\gamma},\hat{v}\right)=&\mathcal{T}(c,s,b,\gamma)=\\
&\frac{\sqrt{\rho -1}\left(c,\frac{\sqrt{\rho }}{\sqrt{\rho -1}}s,\frac{\sqrt{\rho }}{\sqrt{\rho -1}}b,\gamma,1\right)}{\sqrt{\left(\rho -1\right)\left(\|s\|^2+\|b\|^2+1\right)+\sum_{\alpha \in \Gamma_d}{s_\alpha }+\sum_{k=1}^Nb_k}}. \notag
\end{align} 
The inverse of fractional mapping $\mathcal{T}$ is given by
\begin{align} \label{eq7} 
&\left(c,s,b,\gamma\right)={\mathcal{T}}^{-1}\left(\hat{c},\hat{s},\hat{b},\hat{\gamma},\hat{v}\right) = \notag \\ 
& \left(\frac{\hat{c}}{\hat{v}},\frac{\sqrt{\rho -1}\hat{s}}{\sqrt{\rho }\hat{v}},\frac{\sqrt{\rho -1}\hat{b}}{\sqrt{\rho }\hat{v}},\frac{\hat{\gamma}}{\hat{v}}\right). 
\end{align} 
Using fractional mapping \eqref{eq6}, the nonconvex QCQP \eqref{eq5} is transformed into the following FP \eqref{eq8}.
\begin{subequations} \label{eq8}
	\begin{align}
        \underset{\Omega_4}{\min}&\quad \frac{\hat{\gamma}}{\hat{v}} \label{eq8_a} \\ \displaybreak[0]
		\text{s.t.} &\quad -M\hat{v}+\frac{\sqrt{\rho-1}}{\sqrt{\rho}}M\hat{b}_k+\left|y^{(k)}\hat{v}-\sum_{\alpha \in \Gamma_d} \hat{c}_\alpha {x^{(k)}}^{\alpha} \right| \le \hat{\gamma},\quad \forall k \in \{1,\dots,N\}, \label{eq8_b} \\ \displaybreak[0]
        &\quad -\frac{\sqrt{\rho-1}}{\sqrt{\rho}} M \hat{s}_\alpha \le \hat{c}_\alpha \le \frac{\sqrt{\rho-1}}{\sqrt{\rho}} M \hat{s}_\alpha,\quad \forall \alpha \in \Gamma_d, \label{eq8_c} \\ \displaybreak[0]
        &\quad 0 \le \hat{s}_\alpha \le \frac{\sqrt{\rho}}{\sqrt{\rho - 1}} \hat{v},\quad \forall \alpha \in \Gamma_d, \label{eq8_d} \\ \displaybreak[0]
        &\quad 0 \le \hat{b}_k \le \frac{\sqrt{\rho}}{\sqrt{\rho - 1}} \hat{v},\quad\forall k \in \{1,\dots,N\}, \label{eq8_e} \\ \displaybreak[0]
        &\quad \sum_{k=1}^N \hat{b}_k = l_b  \frac{\sqrt{\rho}}{\sqrt{\rho - 1}}, \label{eq8_f} \\ \displaybreak[0]
        &\quad \sum_{\alpha \in \Gamma_d} \hat{s}_\alpha = l_m \frac{\sqrt{\rho}}{\sqrt{\rho - 1}}, \label{eq8_g} \\ \displaybreak[0]
        &\quad \hat{v} = \frac{\sqrt{\rho - 1}}{\sqrt{\rho(m_d + N + 1) - 1}}, \label{eq8_h} \\ \displaybreak[0]
        &\quad \frac{(\rho - 1)(\|\hat{s}\|^2 + \|\hat{b}\|^2)}{\rho} + \frac{\hat{v}}{\sqrt{\rho} \sqrt{\rho - 1}} \left( \sum_{\alpha \in \Gamma_d} \hat{s}_\alpha + \sum_{k=1}^N \hat{b}_k \right) + \hat{v}^2 \le 1, \label{eq8_i} \\ \displaybreak[0]
        &\quad \|\hat{s}\|^2 + \|\hat{b}\|^2 + \hat{v}^2 \ge 1, \label{eq8_j}
	\end{align}
\end{subequations}
\noindent where $\Omega = \left\{ \hat{c} \in \mathbb{R}^{m_d},~\hat{s} \in \mathbb{R}^{m_d},~\hat{b} \in \mathbb{R}^N,~\hat{\gamma} \in \mathbb{R},~\hat{v} \in \mathbb{R} \right\}$.

\subsection{Equivalence and Properties}
We begin by analyzing the complexity of the nonconvex FP \eqref{eq8}. It is important to note that the FP \eqref{eq8} introduces only one nonconvex constraint  and one additional variable as compared to the nonconvex QCQP \eqref{eq5}. The parameter $\rho$ is assumed to be sufficiently large such that constraint (\ref{eq8_i}) becomes convex. Under this assumption, it can be shown that all constraints of the FP \eqref{eq8} are convex except the last one (\ref{eq8_j}) which is scalar and quadratic. This claim is proved in Theorem \ref{thm1}.
\begin{theorem} \label{thm1}
The FP \eqref{eq8} has a convex objective function and convex constraints, except for constraint(\ref{eq8_j}), provided $\rho \ge 1+\frac{\sqrt{m_d+N}}{2}$.
\end{theorem}
\begin{proof} Due to constraint (\ref{eq8_h}), the objective function is linear in $\hat{\gamma}$, since $\hat{v}$ is not a decision variable but rather a given parameter.\\
The convexity of constraints (\ref{eq8_a})-(\ref{eq8_h}) are evident, so it remains to verify the convexity of constraint (\ref{eq8_i}). For this purpose, consider the next expressions (note that, $\rho \ge 1+\frac{\sqrt{m_d+N}}{2}$)
\begin{align} \label{eq9} 
&\frac{\left(\rho -1\right)\left(\left\|\hat{s}\right\|^2+\|\hat{b}\|^2\right)}{\rho }+\frac{\hat{v}\sum_{\alpha \in {\mathrm{\Gamma }}_d}{{\hat{s}}_{\alpha }}+\hat{v}\sum_{k=1}^N \hat{b}_k}{\sqrt{\rho }\sqrt{\rho -1}}+{\hat{v}}^2=\\
&\left[\begin{matrix} \hat{s}\\\hat{b}\\\hat{b}
\end{matrix} \right]^T
\left[\begin{matrix} \frac{\rho-1}{\rho}I_{m_d}& 0& \frac{e_{m_d}}{2\sqrt{\rho}\sqrt{\rho-1}}\\0& \frac{\rho-1}{\rho}I_{N}&\frac{e_{N}}{2\sqrt{\rho}\sqrt{\rho-1}}\\\frac{e_{m_d}^T}{2\sqrt{\rho}\sqrt{\rho-1}}& \frac{e_{N}^T}{2\sqrt{\rho}\sqrt{\rho-1}}&1
\end{matrix} \right]
\left[\begin{matrix} \hat{s}\\\hat{b}\\\hat{b}
\end{matrix} \right], \notag
\end{align} 
\begin{align} \label{eq10} 
&1-\frac{e_{m_d+N}^T}{2\sqrt{\rho }\sqrt{\rho -1}}{\left(\frac{\rho -1\ }{\rho }I_{m_d+N}\right)}^{-1}\frac{e_{m_d+N}}{2\sqrt{\rho }\sqrt{\rho -1}}=1-\frac{m_d+N}{4{\left(\rho -1\ \right)}^2\ }\ge 0.
\end{align} 
Using \eqref{eq10}, matrix $\left[\begin{matrix} \frac{\rho-1}{\rho}I_{m_d}& 0& \frac{e_{m_d}}{2\sqrt{\rho}\sqrt{\rho-1}}\\0& \frac{\rho-1}{\rho}I_{N}&\frac{e_{N}}{2\sqrt{\rho}\sqrt{\rho-1}}\\\frac{e_{m_d}^T}{2\sqrt{\rho}\sqrt{\rho-1}}& \frac{e_{N}^T}{2\sqrt{\rho}\sqrt{\rho-1}}&1
\end{matrix} \right]
$ is positive-definite that implies constraint (\ref{eq8_i}) represents an ellipsoid and is convex. \qed
\end{proof}
Next, Theorem \ref{thm2} proves another fundamental property of the FP \eqref{eq8}. 
\begin{theorem} \label{thm2}
The nonconvex QCQP \eqref{eq5} and FP \eqref{eq8} are equivalent. 
\end{theorem}
\begin{proof}
Suppose $\left(c,s,b,\gamma\right)$ is a feasible solution of QCQP \eqref{eq5} and $\left(\hat{c},\hat{s},\hat{b},\hat{\gamma},\hat{v}\right)=\mathcal{T}\left(c,s,b,\gamma\right)$. Substituting $\left(\hat{c},\hat{s},\hat{b},\hat{\gamma},\hat{v}\right)$ into the constraints of the FP \eqref{eq8} results in (note that ${\left\|s\right\|}^2=\sum_{\alpha \in \Gamma_d}{s_\alpha}$ and $\sum_{k=1}^Nb_k^2=\sum_{k=1}^N b_k$)
\begin{align}  
&-M\hat{v}+\frac{\sqrt{\rho-1}}{\sqrt{\rho}}M\hat{b}_k +\left|y^{\left(k\right)}\hat{v}-\sum_{\alpha \in {\mathrm{\Gamma }}_d}{{\hat{c}}_{\alpha }{x^{\left(k\right)}}^{\alpha }}\right|-\hat{\gamma }= \label{eq11}\\ \displaybreak[0]
&\hat{v}\left(-M(1-b_k)+\left|y^{\left(k\right)}-\sum_{\alpha \in {\mathrm{\Gamma }}_d}{{\hat{c}}_{\alpha }{x^{\left(k\right)}}^{\alpha }}\right|-\gamma \right)\le 0,\quad \forall k\in \left\{1,\dots ,N\right\}, \notag  \\ \displaybreak[0]
&\left|{\hat{c}}_{\alpha}\right|-\frac{\sqrt{\rho-1}}{\sqrt{\rho}}M{\hat{s}}_{\alpha }=\hat{v}\left(\left|c_{\alpha}\right|-Ms_{\alpha }\right)\le 0,\quad \forall \alpha \in {\mathrm{\Gamma }}_d, \label{eq12}\\ \displaybreak[0]
&{\hat{s}}_{\alpha}-\frac{\sqrt{\rho }}{\sqrt{\rho -1}}\hat{v}=\frac{\sqrt{\rho }}{\sqrt{\rho -1}}\hat{v}\left(s_{\alpha}-1\right)\le 0,\quad \forall \alpha \in {\Gamma}_d, \label{eq13}\\ \displaybreak[0]
&\hat{b}_k-\frac{\sqrt{\rho }}{\sqrt{\rho-1}}\hat{v}=\frac{\sqrt{\rho}}{\sqrt{\rho-1}}\hat{v}\left(b_k-1\right)\le 0,\quad \forall k \in \{1,...,N\}, \label{eq14}\\ \displaybreak[0]
& \sum_{k=1}^N\hat{b}_k-l_b\frac{\sqrt{\rho}}{\sqrt{\rho-1}}\hat{v}= \frac{\sqrt{\rho}}{\sqrt{\rho-1}}\hat{v}\left(\sum_{k=1}^Nb_k-l_b\right)=0, \label{eq15}
\end{align}
\begin{align}
& \sum_{\alpha \in \Gamma_d}{{\hat{s}}_{\alpha}}-l_m\frac{\sqrt{\rho}}{\sqrt{\rho-1}}\hat{v}= \frac{\sqrt{\rho}}{\sqrt{\rho-1}}\hat{v}\left(\sum_{\alpha \in\Gamma_d}{s_{\alpha }}-l_m\right)=0, \label{eq16}\\ \displaybreak[0]
&\frac{\left(\rho -1\right)\left(\|\hat{s}\|^2+\|\hat{b}\|^2\right)}{\rho}+\frac{\hat{v}\sum_{\alpha \in \Gamma_d}\hat{s}_\alpha+\hat{v}\sum_{k=1}^N\hat{b}_k}{\sqrt{\rho}\sqrt{\rho-1}}+\hat{v}^2= \label{eq17}\\ \displaybreak[0]
&{\hat{v}}^2\ \left(\|s\|^2+\|b\|^2+\frac{1}{\rho -1}\left(\sum_{\alpha \in \Gamma_d}s_\alpha+\sum_{k=1}^N b_k\right) +1\right)\le 1, \notag \\ \displaybreak[0]
&{\|\hat{s}\|}^2+\|\hat{b}\|^2+{\hat{v}}^2=\frac{\rho \|s\|^2+\rho\|b\|^2+\rho -1}{\left(\rho -1\right)\left(\|s\|^2+\|b\|^2+1\right)+\sum_{\alpha \in \Gamma_d}{s_{\alpha}}+\sum_{k=1}^Nb_k}\ge1. \label{eq18} \displaybreak[0]
\end{align}
Based on expressions \eqref{eq11}-\eqref{eq18}, the transformed point  $\left(\hat{c},\hat{s},\hat{b},\hat{\gamma},\hat{v}\right)=\mathcal{T}\left(c,s\right)$ satisfies all constraints of the FP  \eqref{eq8}.\\
Now, suppose $\left(\hat{c},\hat{s},\hat{b},\hat{\gamma},\hat{v}\right)$ is feasible for \eqref{eq8} and $\left(c,s,b,\gamma\right)={\mathcal{T}}^{-1}\left(\hat{c},\hat{s},\hat{b},\hat{\gamma},\hat{v}\right)$. Then we can derive
\begin{align}
&-M(1-b_k)+\left|y^{(k)}-\sum_{\alpha \in \Gamma_d}{c_{\alpha}{x^{(k)}}^{\alpha }}\right|-\gamma = \notag \\ \displaybreak[0]
&\frac{1}{\hat{v}}\left(-M\hat{v}+\frac{\sqrt{\rho-1}}{\sqrt{\rho}}M\hat{b}_k+\left|y^{(k)}\hat{v}-\sum_{\alpha \in \Gamma_d}{{\hat{c}}_{\alpha}{x^{(k)}}^{\alpha}}\right|-\hat{\gamma}\right) \le 0,\quad \forall k\in \{1,\dots,N\}, \label{eq19} \\ \displaybreak[0]
&\left|c_{\alpha}\right|-Ms_{\alpha} = \frac{1}{\hat{v}}\left(\left|{\hat{c}}_{\alpha}\right|-\frac{\sqrt{\rho-1}}{\sqrt{\rho}}M{\hat{s}}_{\alpha}\right) \le 0,\quad \forall \alpha \in \Gamma_d, \label{eq20}\\ \displaybreak[0]
&\sum_{k=1}^N b_k - l_b = \frac{\sqrt{\rho-1}}{\sqrt{\rho}\hat{v}}\left(\sum_{k=1}^N \hat{b}_k - l_b\frac{\sqrt{\rho}}{\sqrt{\rho-1}}\hat{v}\right) = 0, \label{eq21} \\ \displaybreak[0]
&\sum_{\alpha \in \Gamma_d} s_{\alpha} - l_m = \frac{\sqrt{\rho-1}}{\sqrt{\rho}\hat{v}}\left(\sum_{\alpha \in \Gamma_d} \hat{s}_{\alpha} - l_m\frac{\sqrt{\rho}}{\sqrt{\rho-1}}\hat{v}\right) = 0, \label{eq22} \\ \displaybreak[0]
&\sum_{\alpha \in \Gamma_d} s_{\alpha} + \sum_{k=1}^N b_k - \sum_{\alpha \in \Gamma_d} s_\alpha^2 - \sum_{k=1}^N b_k^2 =\frac{\rho - 1}{\hat{v}}\times \notag \\ \displaybreak[0]
&\left( \frac{(\rho - 1)\|\hat{s}\|^2 + \|\hat{b}\|^2}{\rho} + \frac{\hat{v} \sum_{\alpha \in \Gamma_d} \hat{s}_\alpha + \hat{v} \sum_{k=1}^N \hat{b}_k}{\sqrt{\rho} \sqrt{\rho - 1}} + \hat{v}^2 - \left(\|\hat{s}\|^2 + \|\hat{b}\|^2 + \hat{v}^2\right) \right) \notag \\& \le 0, \label{eq23} \\ \displaybreak[0]
&s_{\alpha} - 1 = \frac{\sqrt{\rho-1}}{\sqrt{\rho}\hat{v}}\left(\hat{s}_{\alpha} - \frac{\sqrt{\rho}}{\sqrt{\rho - 1}}\hat{v} \right) \le 0,\quad \forall \alpha \in \Gamma_d, \label{eq24} 
\end{align}
\begin{align}
&b_k - 1 = \frac{\sqrt{\rho-1}}{\sqrt{\rho}\hat{v}}\left(\hat{b}_k - \frac{\sqrt{\rho}}{\sqrt{\rho - 1}}\hat{v} \right) \le 0,\quad \forall k \in \{1,\dots,N\}. \label{eq25}
\end{align} 
According to expression \eqref{eq19}-\eqref{eq25}, the point $\left(c,s,b,\gamma\right)={\mathcal{T}}^{-1}\left(\hat{c},\hat{s},\hat{b},\hat{\gamma},\hat{v}\right)$ belongs to the feasible region of the QCQP \eqref{eq5}. Therefore, the QCQP \eqref{eq5} and the FP \eqref{eq8} are equivalent in terms of both feasible sets and objective values. \qed
\end{proof} 
Since the QCQP and FP models are equivalent, we focus on solving the FP~\eqref{eq8} in the rest of this paper. 
\section{Solution Algorithm} \label{sec:salg}
This section presents a tractable approach for solving the FP \eqref{eq8} using a proposed relaxation technique. To address the nonconvex constraint (\ref{eq8_j}), first, we use the Semi-Definite Cone (SDC) and the Second-Order Cone (SOC) to convexify this constraint. Then these two convexification techniques are used as two baselines for our proposed convexification technique where the nonconvex constraint (\ref{eq8_j}) is linearized. We theoretically prove that our convexification technique is better than those baselines and adapt it for our proposed TS-CRR algorithm. The TS-CRR algorithm extracts the sparse polynomial regression structure from the relaxed solution.

We begin by using SDC to exactly reformulate the nonconvex constraint (\ref{eq8_j}) as constraints (\ref{eq26_j}), (\ref{eq26_k}), and (\ref{eq26_l}).
\begin{subequations} \label{eq26}
	\begin{align}
        \underset{\Omega_5}{\text{min}}&\quad~\frac{\hat{\gamma}}{\hat{v}}, \label{eq26_a} \\ \displaybreak[0]
		\text{s.t.}&\quad -M\hat{v}+\frac{\sqrt{\rho-1}}{\sqrt{\rho}}M\hat{b}_k+\left|y^{(k)}\hat{v}-\sum_{\alpha \in \Gamma_d} \hat{c}_\alpha {x^{(k)}}^{\alpha} \right| \le \hat{\gamma},\quad \forall k \in \{1,\dots,N\}, \label{eq26_b} \\ \displaybreak[0]
        &\quad -\frac{\sqrt{\rho-1}}{\sqrt{\rho}} M \hat{s}_\alpha \le \hat{c}_\alpha \le \frac{\sqrt{\rho-1}}{\sqrt{\rho}} M \hat{s}_\alpha,\quad \forall \alpha \in \Gamma_d, \label{eq26_c} \\ \displaybreak[0]
        &\quad 0 \le \hat{s}_\alpha \le \frac{\sqrt{\rho}}{\sqrt{\rho - 1}} \hat{v},\quad \forall \alpha \in \Gamma_d, \label{eq26_d} \\ \displaybreak[0]
        &\quad 0 \le \hat{b}_k \le \frac{\sqrt{\rho}}{\sqrt{\rho - 1}} \hat{v},\quad \forall k \in \{1,\dots,N\}, \label{eq26_e} \\ \displaybreak[0]
        &\quad \sum_{k=1}^N \hat{b}_k = l_b \frac{\sqrt{\rho}}{\sqrt{\rho - 1}}, \label{eq26_f} \\ \displaybreak[0]
        &\quad \sum_{\alpha \in \Gamma_d} \hat{s}_\alpha = l_m  \frac{\sqrt{\rho}}{\sqrt{\rho - 1}}, \label{eq26_g} \\ \displaybreak[0]
        &\quad \hat{v} = \frac{\sqrt{\rho - 1}}{\sqrt{\rho(m_d + N + 1) - 1}}, \label{eq26_h} \\ \displaybreak[0]
        &\quad \frac{(\rho - 1)(\|\hat{s}\|^2 + \|\hat{b}\|^2)}{\rho} + \frac{\hat{v}}{\sqrt{\rho} \sqrt{\rho - 1}} \left( \sum_{\alpha \in \Gamma_d} \hat{s}_\alpha + \sum_{k=1}^N \hat{b}_k \right) + \hat{v}^2 \le 1, \label{eq26_i} \\ \displaybreak[0]
        &\quad \operatorname{trace}(\hat{\chi}) \ge 1, \label{eq26_j} \\ \displaybreak[0]
        &\quad \hat{\chi} \succeq 
        \begin{bmatrix}
            \hat{s} \\
            \hat{b} \\
            \hat{v}
        \end{bmatrix}
        \begin{bmatrix}
            \hat{s} \\
            \hat{b} \\
            \hat{v}
        \end{bmatrix}^T, \label{eq26_k} \\ \displaybreak[0]
        &\quad \operatorname{rank}(\hat{\chi}) = 1, \label{eq26_l}
	\end{align}
\end{subequations}
\noindent where $\Omega_5 = \left\{ \hat{c} \in \mathbb{R}^{m_d},~\hat{s} \in \mathbb{R}^{m_d},~\hat{b} \in \mathbb{R}^N,~\hat{\gamma} \in \mathbb{R},~\hat{v} \in \mathbb{R},~\hat{\chi} \in \mathbb{R}^{m_d+N+1} \right\}$ is the set of variables. Note that constraints (\ref{eq26_k}) and (\ref{eq26_l}) enforce matrix $\hat{\chi}$ to be equivalent to $\left[ \begin{matrix}			\hat{s}\\\hat{b}\\\hat{v}\end{matrix}
			\right]\left[ \begin{matrix}			\hat{s}\\\hat{b}\\\hat{v}\end{matrix}
			\right]^T$. Thus, constraint (\ref{eq26_k}) equivalently represents the nonconvex constraint (\ref{eq8_j}). We omit the rank one constraint from \eqref{eq26} to obtain the SDC-based relaxation of FP \eqref{eq8}. In the SOC-based relaxation, we replace the SDC constraint (\ref{eq26_k}) with the SOC constraint (\ref{eq28_l}) below 
\begin{align}\label{eq28_l}
&e_i^\top \left( \hat{\chi} - \begin{bmatrix} \hat{s} \\ \hat{b} \\ \hat{v} \end{bmatrix} \begin{bmatrix} \hat{s} \\ \hat{b} \\ \hat{v} \end{bmatrix}^\top \right) e_i e_j^\top \left( \hat{\chi} - \begin{bmatrix} \hat{s} \\ \hat{b} \\ \hat{v} \end{bmatrix} \begin{bmatrix} \hat{s} \\ \hat{b} \\ \hat{v} \end{bmatrix}^\top \right) e_j \notag \\ 
&\ge \left( e_i^\top \left( \hat{\chi} - \begin{bmatrix} \hat{s} \\ \hat{b} \\ \hat{v} \end{bmatrix} \begin{bmatrix} \hat{s} \\ \hat{b} \\ \hat{v} \end{bmatrix}^\top \right) e_j \right)^2,\quad \forall i \ne j \in \{1,\dots,m_d+N+1\},
\end{align} 
\noindent where $e_i$ is the $i^{th}$ column of the identity matrix. 

In our proposed linear-based relaxation, the nonconvex constraint \eqref{eq8_j} is linearized and replaced by the following linear constraint 
\begin{align} \label{eq29_j}
\sum_{\alpha \in \Gamma_d} \hat{s}_\alpha + \sum_{k=1}^N \hat{b}_k + \hat{v} \ge 1. 
\end{align}
Specifically, the squared terms in \eqref{eq8_j} are replaced with their linear counterparts, as shown in constraint \eqref{eq29_j}. 

Theorem~\ref{thm3} below proves that our proposed linear-based relaxation provides a valid relaxation of the original FP~\eqref{eq8}, ensuring that any feasible solution to the original problem remains feasible to the relaxed formulation problem. 
\begin{theorem} \label{thm3}
	All feasible solutions of FP \eqref{eq8} belongs to the feasible region of the linear-based relaxation problem.
\end{theorem}
\begin{proof}Let $\left(\hat{c},\hat{s},\hat{b},\hat{\gamma},\hat{v}\right)$ be an arbitrary selected feasible solution of the FP \eqref{eq8} which implies:
\begin{align}
&-M\hat{v}+\frac{\sqrt{\rho-1}}{\sqrt{\rho}}M\hat{b}_k+\left|y^{(k)}\hat{v}-\sum_{\alpha \in \Gamma_d} \hat{c}_{\alpha} {x^{(k)}}^{\alpha} \right| \le \hat{\gamma},\quad \forall k \in \{1,\dots,N\}, \label{eq30} \\ \displaybreak[0]
&-\frac{\sqrt{\rho -1}}{\sqrt{\rho}}M\hat{s}_{\alpha} \le \hat{c}_{\alpha} \le \frac{\sqrt{\rho -1}}{\sqrt{\rho}}M\hat{s}_{\alpha},\quad \forall \alpha \in \Gamma_d, \label{eq31} \\ \displaybreak[0]
&0 \le \hat{s}_{\alpha} \le \frac{\sqrt{\rho}}{\sqrt{\rho -1}}\hat{v},\quad \forall \alpha \in \Gamma_d, \label{eq32} \\ \displaybreak[0]
&0 \le \hat{b}_k \le \frac{\sqrt{\rho}}{\sqrt{\rho -1}}\hat{v},\quad \forall k \in \{1,\dots,N\}, \label{eq33} \\ \displaybreak[0]
&\sum_{k=1}^N \hat{b}_k = l_b\frac{\sqrt{\rho}}{\sqrt{\rho -1}}\hat{v}, \label{eq34} \\ \displaybreak[0]
&\sum_{\alpha \in \Gamma_d} \hat{s}_\alpha = l_m\frac{\sqrt{\rho}}{\sqrt{\rho -1}}\hat{v}, \label{eq35} \\ \displaybreak[0]
&\frac{(\rho - 1)(\|\hat{s}\|^2 + \|\hat{b}\|^2)}{\rho} + \frac{\hat{v} \sum_{\alpha \in \Gamma_d} \hat{s}_{\alpha} + \hat{v} \sum_{k=1}^N \hat{b}_k}{\sqrt{\rho} \sqrt{\rho - 1}} + \hat{v}^2 \le 1, \label{eq36} \\ \displaybreak[0]
&\|\hat{s}\|^2 + \|\hat{b}\|^2 + \hat{v}^2 \ge 1. \label{eq37}
\end{align}

According to (\ref{eq30})-(\ref{eq37}), the point $\left(\hat{c},\hat{s},\hat{b},\hat{\gamma},\hat{v}\right)$ satisfies constraints (\ref{eq26_b})-(\ref{eq26_i}). Thus, it suffices to show that \eqref{eq29_j} holds in order to complete the proof. To this end, we consider the following covering cases:

\noindent \textbf{Case 1.} Assume all members of ${\left\{{\hat{s}}_{\alpha }\right\}}_{\alpha \in {\mathrm{\Gamma }}_d}$, all members of $\{\hat{b}_k\}_{k=1}^N$, and $\hat{v}$ are smaller than 1. This assumption results in

\begin{align}
&\hat{s}_\alpha \ge \hat{s}_\alpha^2~\forall \alpha \in \Gamma_d, \label{eq38} \\\displaybreak[0]
&\hat{b}_k \ge \hat{b}_k^2~\forall k \in \{1,\dots,N\}, \label{eq39} \\\displaybreak[0]
&\hat{v} \ge \hat{v}^2. \label{eq40}
\end{align}

Using \eqref{eq38}-\eqref{eq40}, we obtain
\begin{equation} \label{eq41} 
\|\hat{s}\|^2+\|\hat{b}\|^2+{\hat{v}}^2\le \sum_{\alpha \in \Gamma_d}{\hat{s}_\alpha }+\sum_{k=1}^N\hat{b}_k+\hat{v}.
\end{equation} 
Combining \eqref{eq37} and \eqref{eq41}, it follows directly that $\sum_{\alpha \in \Gamma_d}{{\hat{s}}_{\alpha }}+\sum_{k=1}^N\hat{b}_k+\hat{v}\ge 1$ which implies that the point $\left(\hat{c},\hat{s},\hat{b},\hat{\gamma},\hat{v}\right)$ satisfies constraint (\ref{eq29_j}).\\
\noindent \textbf{Case 2.} Assume there exists an entry in the set $\left\{{\left\{{\hat{s}}_{\alpha }\right\}}_{\alpha \in \Gamma_d},\{\hat{b}_k\}_{k=1}^N,\hat{v}\right\}$ that is greater than or equal to 1. Without loss of the generality, suppose $\hat{v}\ge 1$ which directly implies $\sum_{\alpha \in \Gamma_d}{{\hat{s}}_{\alpha }}+\sum_{k=1}^N\hat{b}_k+\hat{v}\ge 1$. Therefor, the point $\left(\hat{c},\hat{s},\hat{b},\hat{\gamma},\hat{v}\right)$ satisfies constraint (\ref{eq29_j}) in this case as well. \qed 
\end{proof}

\begin{theorem} \label{thm4_2}
The linear-based relaxation of FP \eqref{eq8} is tighter than the SDC-based relaxation of FP \eqref{eq8}.
\end{theorem}
\begin{proof}
Let $\left(\hat{c}_l, \hat{s}_l, \hat{b}_l, \hat{\gamma}_l, \hat{v}_l\right)$ be an arbitrary feasible solution to the linear-based relaxation of FP \eqref{eq8}. Then, by definition, the following inequality must hold
\begin{equation} \label{eq:sum1}
\sum_{\alpha \in \Gamma_d} \hat{s}_{l,\alpha} + \sum_{k=1}^N \hat{b}_{l,k} + \hat{v}_l \ge 1.
\end{equation}

We now prove that this point also belongs to the feasible region of the SDC-based relaxation by constructing a suitable matrix $\hat{\chi}_l$ and verifying the SDC constraints.

\textbf{Case 1:} Assume $\|\hat{s}_l\|^2 + \|\hat{b}_l\|^2 + \hat{v}_l^2 \ge 1$. Let
\begin{equation}
\hat{\chi}_l = 
\begin{bmatrix}
\hat{s}_l \\
\hat{b}_l \\
\hat{v}_l
\end{bmatrix}
\begin{bmatrix}
\hat{s}_l \\
\hat{b}_l \\
\hat{v}_l
\end{bmatrix}^\top.
\end{equation}

Then, $\hat{\chi}_l \succeq \begin{bmatrix} \hat{s}_l \\ \hat{b}_l \\ \hat{v}_l \end{bmatrix} \begin{bmatrix} \hat{s}_l \\ \hat{b}_l \\ \hat{v}_l \end{bmatrix}^\top$ holds trivially and $\operatorname{trace}(\hat{\chi}_l)=\|\hat{s}_l\|^2 + \|\hat{b}_l\|^2 + \hat{v}_l^2 \ge 1$ by the case's assumption. Thus, tuple $(\hat{c}_l, \hat{s}_l, \hat{b}_l, \hat{\gamma}_l, \hat{v}_l, \hat{\chi}_l)$ satisfies the SDC constraints \eqref{eq26_k} and \eqref{eq26_j}, and is a feasible point of the SDC-based relaxation.

\textbf{Case 2:} Assume $\|\hat{s}_l\|^2 + \|\hat{b}_l\|^2 + \hat{v}_l^2 < 1$. Define
\begin{equation} \label{eq143}
z = \begin{bmatrix}
\hat{s}_l \\
\hat{b}_l \\
\hat{v}_l
\end{bmatrix}, \quad
e = \mathbf{1} \in \mathbb{R}^{m_d + N + 1}, \quad
\hat{\chi}_l = \frac{1}{2}(z e^\top + e z^\top).
\end{equation}

Due to (\ref{eq26_d}), (\ref{eq26_e}), and(\ref{eq26_h}), all entries in $z$ are non-negative. Using this fact and inequality $\|\hat{s}_l\|^2 + \|\hat{b}_l\|^2 + \hat{v}_l^2 < 1$, the following relation can be directly obtained.
\begin{equation} \label{eq144}
\sum_{i=1}^{m_d+N+1} z_i^2 \le \sum_{i=1}^{m_d+N+1} z_i \quad \Rightarrow \quad \|z\|^2 \le \sum_{i=1}^{m_d+N+1} z_i.    
\end{equation}

Then $\hat{\chi}_l$ is symmetric by construction. We now show that $\hat{\chi}_l \succeq zz^\top$ by verifying that
\begin{equation}
y^\top \hat{\chi}_l y \ge y^\top zz^\top y, \quad \forall y \in \mathbb{R}^{m_d + N + 1}.
\end{equation}

(a) If $y \perp z$, then $y^T z = 0$, and the inequality
$y^\top \hat{\chi}_l y = \frac{1}{2}(y^\top z)(e^\top y) + \frac{1}{2}(y^\top e)(z^\top y) = 0$  holds with equality.

(b) If $y = z$, then $y^\top \hat{\chi}_l y = \frac{1}{2}(z^\top z)(e^\top z) + \frac{1}{2}(z^\top e)(z^\top z) = (z^\top z)(e^\top z) = \|z\|^2 \cdot \sum_i z_i\ge \|z\|^4=y^T(zz^T)y$ based on (\ref{eq144}). 

Together, these cases show that $\hat{\chi}_l \succeq zz^\top$, i.e., SDC constraint \eqref{eq26_k} holds. Finally, by the linear constraint \eqref{eq:sum1}, we have $\operatorname{trace}(\hat{\chi}_l) = z^\top e = \sum_{\alpha \in \Gamma_d} \hat{s}_{l,\alpha} + \sum_{k=1}^N \hat{b}_{l,k} + \hat{v}_l \ge 1
$. Therefore, constraint \eqref{eq26_j} is also satisfied. Thus, the tuple $\left(\hat{c}_l, \hat{s}_l, \hat{b}_l, \hat{\gamma}_l, \hat{v}_l, \hat{\chi}_l\right)$ is feasible for the SDC-based relaxation. 

Since this holds for any feasible point of the linear-based relaxation, it is tighter than the SDC-based relaxation. \qed
\end{proof}

\begin{corollary} \label{crl1}
The linear-based relaxation of FP \eqref{eq8} is tighter than the SOC-based relaxation of FP \eqref{eq8}.
\end{corollary}
\begin{proof}
As shown earlier, the SOC-based relaxation is obtained by relaxing the SDC constraint (in the SDC-based relaxation) into a set of SOC constraints. Since the SDC-based relaxation is already weaker than the linear-based relaxation (Theorem~\ref{thm4_2}), and the SOC-based relaxation further relaxes the SDC-based relaxation model, it follows that the linear-based relaxation is also tighter than the SOC-based relaxation. \qed
\end{proof}

Between the three convex relaxations for FP \eqref{eq8}, the linear relaxation best balances solution quality and computational efficiency. As formally proven, the feasible region of the linear-based relaxation is strictly contained within the feasible region of the SDC-based relaxation, thereby establishing its superior tightness. Furthermore, the linear-based relaxation replaces the original nonconvex quadratic constraint \eqref{eq8_j} with a single linear inequality \eqref{eq29_j}, making it significantly simpler than both the SDC- and SOC-based relaxations. Therefore, the linear-based relaxation is not only the tightest among the three but also the most tractable from an algorithmic perspective. 

Our proposed Two-Step Convex Relaxation and Recovery (TS-CRR) algorithm first solves the linear-based relaxation of the FP \eqref{eq8} to identify the optimal monomials. It then uses these monomials to compute their corresponding coefficients by solving a recovery LP. The steps of the TS-CRR algorithm are outlined below:
\begin{algorithm}[H]
\caption{Two-Step Convex Relaxation and Recovery (TS-CRR)}
\begin{algorithmic}[1]
\State \textbf{Input:} Data $\{x^{(k)}, y^{(k)}\}_{k=1}^N$,
\State \textit{Step 1: Optimal monomial}: Solve the linear-based convex relaxation of FP \eqref{eq8} to obtain $\left(\hat{c},\hat{s},\hat{b},\hat{\gamma},\hat{v}\right)$,
\State Compute $\left(s^*,b^*,\gamma^*\right) = \mathcal{T}^{-1}(\hat{s},\hat{b},\hat{\gamma},\hat{v})$,
\State \textit{Step 2: Coefficient recovery}: Solve LP \eqref{eq42} to recover the coefficients,
\begin{subequations} \label{eq42}
\begin{align}
\min_{c,\gamma} \quad & \gamma \\
\text{s.t.} \quad & \left| y^{(k)} - \sum_{\alpha \in \Gamma_d} s^*_\alpha c_\alpha (x^{(k)})^\alpha \right| \le \gamma, \quad \forall k.
\end{align}
\end{subequations} 
\State \textbf{Output:} Sparse polynomial regression coefficients $c^*$.
\end{algorithmic}
\end{algorithm}

The TS-CRR algorithm, in step 1, finds optimal monomials and then in step 2 optimally determine the corresponding coefficients. 

\begin{theorem} \label{thm:lp_exact}
Let \( (\hat{c}^*, \hat{s}^*, \hat{b}^*, \hat{\gamma}^*, \hat{v}^*) \) be an optimal solution to the linear-based relaxation of the FP \eqref{eq8}. If
\begin{equation} \label{eq:lp_exact_condition}
\|\hat{s}^*\|^2 + \|\hat{b}^*\|^2 + (\hat{v}^*)^2 = \sum_{\alpha \in \Gamma_d} \hat{s}_\alpha^* + \sum_{k=1}^N \hat{b}_k^* + \hat{v}^*,
\end{equation}
then the linear-based relaxation solves FP \eqref{eq8} exactly. That is, \( (\hat{c}^*, \hat{s}^*, \hat{b}^*, \hat{\gamma}^*, \hat{v}^*) \) is also optimal for the original FP \eqref{eq8}.
\end{theorem}

\begin{proof}
Constraint \eqref{eq8_j} is the only nonconvex constraint, which is replaced by the linear inequality, in the linear-based relaxation:
\begin{equation} \label{eq:linrelax}
\sum_{\alpha \in \Gamma_d} \hat{s}_\alpha + \sum_{k=1}^N \hat{b}_k + \hat{v} \ge 1.
\end{equation}

As already established in Theorem~\ref{thm4_2}, the inequality
\begin{equation} \label{eq49:ineq}
\|\hat{s}\|^2 + \|\hat{b}\|^2 + \hat{v}^2 \le \sum_{\alpha \in \Gamma_d} \hat{s}_\alpha + \sum_{k=1}^N \hat{b}_k + \hat{v},
\end{equation}
holds for all \( \hat{s}, \hat{b}, \hat{v} \in [0,1] \), which directly implies that every feasible point of the FP \eqref{eq8} satisfies the linear constraint \eqref{eq:linrelax}. Therefore, the linear-based relaxation defines a larger feasible region than the feasible region of the FP \eqref{eq8}. 

Now, suppose that the linear-based relaxation returns an optimal solution \( (\hat{c}^*, \hat{s}^*, \hat{b}^*, \hat{\gamma}^*, \hat{v}^*) \) satisfying strict equality in \eqref{eq49:ineq}, i.e.,
\begin{equation} \label{eq:equality}
\|\hat{s}^*\|^2 + \|\hat{b}^*\|^2 + (\hat{v}^*)^2 = \sum_{\alpha \in \Gamma_d} \hat{s}_\alpha^* + \sum_{k=1}^N \hat{b}_k^* + \hat{v}^*.
\end{equation}

This implies that
\begin{equation}
(\hat{s}_\alpha^*)^2 = \hat{s}_\alpha^*, \quad (\hat{b}_k^*)^2 = \hat{b}_k^*, \quad (\hat{v}^*)^2 = \hat{v}^*,
\end{equation}
for all \( \alpha\) and \( k \), which can only happen if 
\begin{equation}
\hat{s}_\alpha^* \in \{0,1\}, \quad \hat{b}_k^* \in \{0,1\}, \quad \hat{v}^* = 1.
\end{equation}

Therefore, the point \( (\hat{c}^*, \hat{s}^*, \hat{b}^*, \hat{\gamma}^*, \hat{v}^*) \) also satisfies the nonconvex constraint \eqref{eq8_j} of the original FP \eqref{eq8}. Since it is already optimal for the linear-based relaxation and feasible for FP \eqref{eq8}, it must also be optimal for FP \eqref{eq8}.
\qed
\end{proof}

\vspace{6cm}

\section{Numerical Experiments} \label{sec:app}
In this section, the proposed TS-CRR algorithm for sparse polynomial regression with anomalous data filtering is applied and tested on two different data sets. The first dataset consists of electricity prices and the other contains the temperature measurements. 

\subsection{Application to electricity prices}
\noindent To evaluate the performance of the TS-CRR algorithm, we apply it to a dataset of electricity prices. This application is motivated by the high volatility and also the non-stationary nature of electricity prices \cite{bunn2004modelling}. 

The twelve price areas of the Nordic electricity market for the year 2013 (due to its public availability) are considered for our numerical evaluation\footnote{Twelve areas are NO1, NO2, NO3, NO4, NO5, SE1, SE2, SE3, SE4, FI, DK1, and DK2. In the model, the price in each area depends on the other areas’ prices, as well as wind generation and installed capacities across all areas.}. The results of TS-CRR algorithm are compared with several existing models including regression models and artificial intelligence models. The linear and polynomial regression models from category one are selected. Binary decision tree and ensemble regression model are selected from the third category.

Interpolation and extrapolation errors of the selected benchmark models and the TS-CRR algorithm are used as performance indicators. To compare the interpolation performances, we divide the given data based on odd and even hours and exploit odd-hour data to train the models and even-hour data to test the interpolation errors. The extrapolation performances of the models are evaluated using the first 97\% of each month data for training and the remaining 3\% for testing. In fact, approximately the last day of each month is reserved for checking the extrapolation performance of each model.

\subsubsection{Results of the TS-CRR algorithm}
First, we present results of TS-CRR algorithm. The results consist of the regression performance in all months of 2013 for 12 areas.  
\begin{table}[h!]
	\centering    
\caption{Parameters of TS-CRR algorithm in electricity-price application.}
	\label{tbl1}
	\begin{tabular}{|c|l|c|}
		\hline
		\textbf{Parameter} & \textbf{Description} & \textbf{Value} \\\hline
		$d$     & Degree of the sparse polynomial     & 2    \\\hline
		$N$     & Number of training data points      & 360  \\\hline
		$m_d$   & Total number of candidate monomials & 1485 \\\hline
		$N - l_b$ & Number of anomalous data points     & 4    \\\hline
		$l_m$   & Number of selected monomials  & 120 \\\hline
	\end{tabular}
\end{table}
 
 The parameters of the TS-CRR algorithm are given in Table \ref{tbl1}. Figs. \ref{Fig1} and Fig. \ref{Fig2} show the R squared values of the TS-CRR algorithm for all price areas and months. These values are separated based on the interpolation and extrapolation parts. 

\noindent
\begin{minipage}[t]{0.48\linewidth}
    \centering    \includegraphics[scale=0.325]{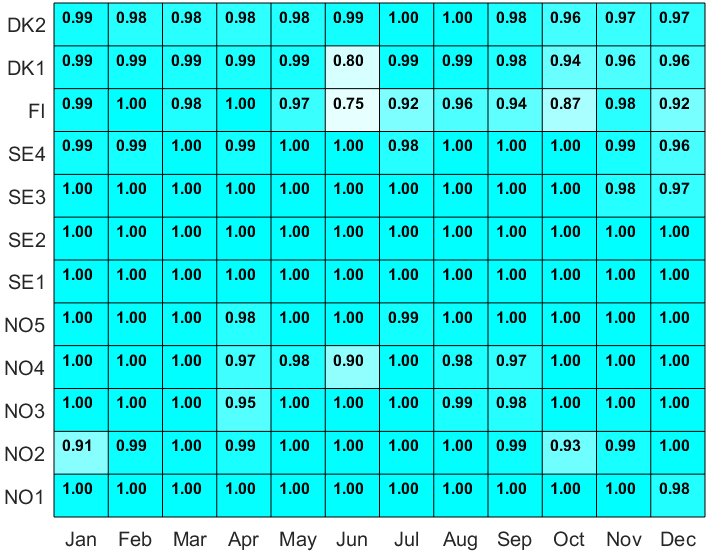}
    \captionof{figure}{R squared values of interpolation performance of the TS-CRR algorithm}
    \label{Fig1}
\end{minipage}
\hfill
\begin{minipage}[t]{0.48\linewidth}
    \centering    \includegraphics[scale=0.325]{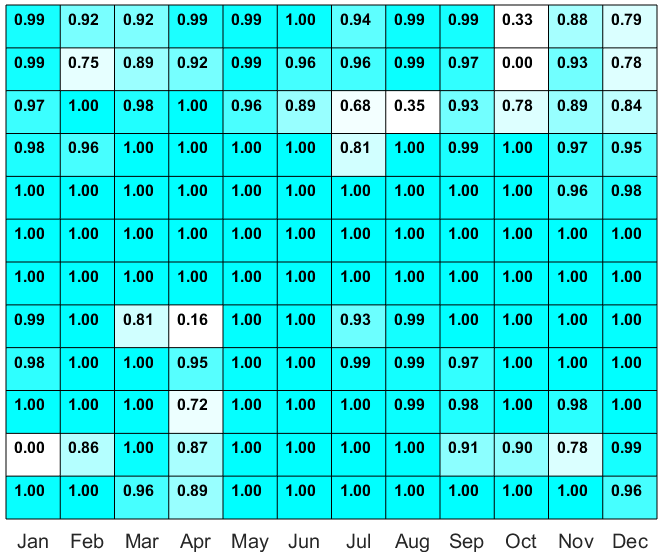}
    \captionof{figure}{R squared values of extrapolation performance of the TS-CRR algorithm}
    \label{Fig2}
\end{minipage}

Based on Figs. \ref{Fig1} and \ref{Fig2}, the proposed TS-CRR algorithm can accurately predict the absent parts of the data. Both interpolation and extrapolation performances of the TS-CRR algorithm are generally acceptable in all the cases studied, with only few exceptions such as the interpolation performances of DK1, FI, and NO4 in June 2013. 

To show the performance of the proposed TS-CRR algorithm, the actual and estimated prices for NO1 in January and March 2013 are shown in Figs. \ref{Fig3} and \ref{Fig4}.

\noindent
\begin{minipage}{\linewidth}
\centering
\begin{minipage}{0.47\linewidth}
\begin{figure} [H]
\centering
     \includegraphics[scale=0.3]{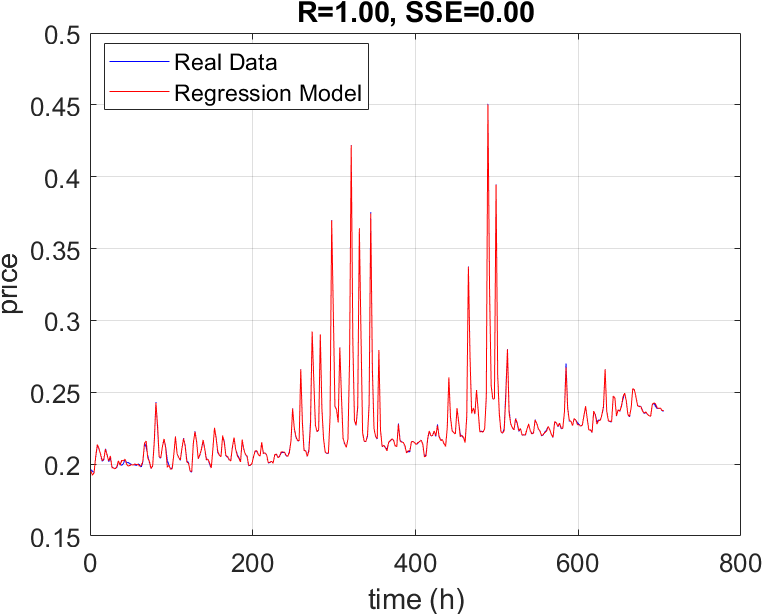}
    \caption{The actual and estimated interpolation prices of area NO1 obtained by TS-CRR in March 2013}
     \label{Fig3}
\end{figure}
\end{minipage}
\hspace{0.01\linewidth}
\begin{minipage}{0.47\linewidth}
\begin{figure} [H]
\centering
     \includegraphics[scale=0.3]{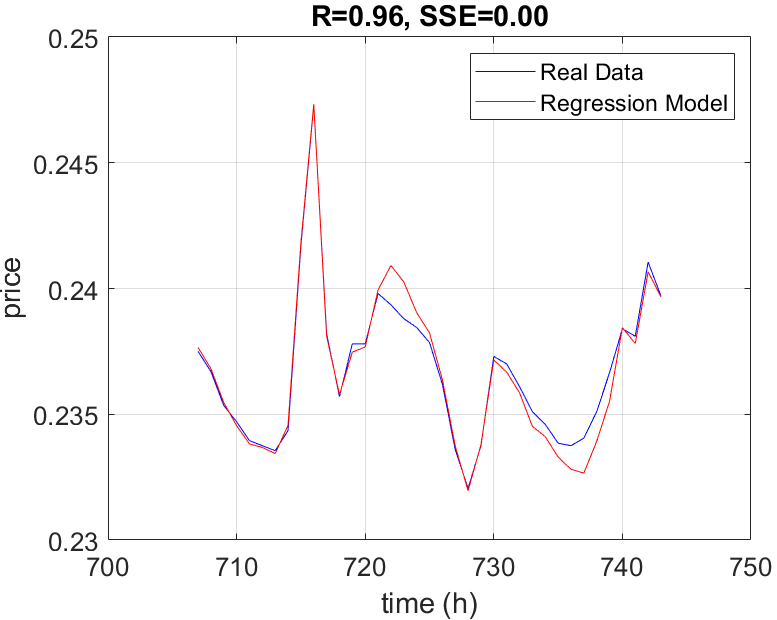}
    \caption{The actual and estimated extrapolation prices of area NO1 obtained by TS-CRR in March 2013}
     \label{Fig4}
\end{figure}
\end{minipage}
\end{minipage}

\subsubsection{Comparison with regression models}
In this subsection, we compare the results of the TS-CRR algorithm with those obtained from linear and polynomial regression models \cite{yao2014new}. The R-squared values for the linear and polynomial regressions are reported in Fig. \ref{Fig5} to Fig. \ref{Fig8}.

\noindent
\begin{minipage}[h]{0.48\linewidth}
    \centering
    \includegraphics[scale=0.325]{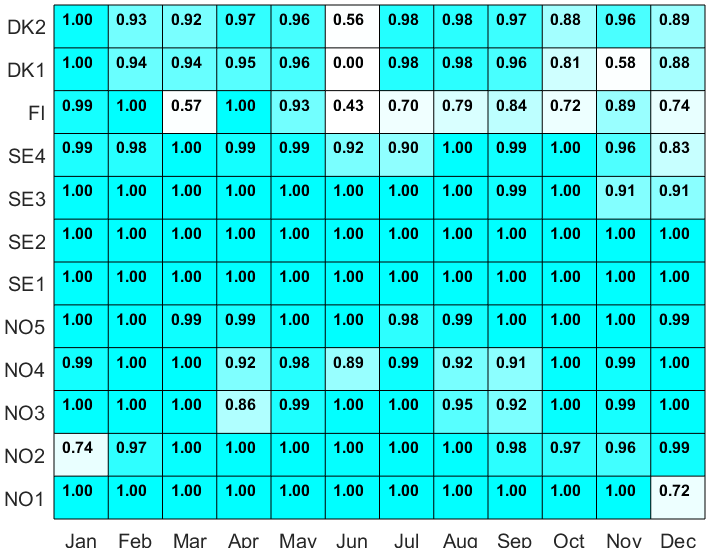}
    \captionof{figure}{R squared values of interpolation performance of the linear regression model}
    \label{Fig5}
\end{minipage}
\hfill
\begin{minipage}[h]{0.48\linewidth}
    \centering    \includegraphics[scale=0.325]{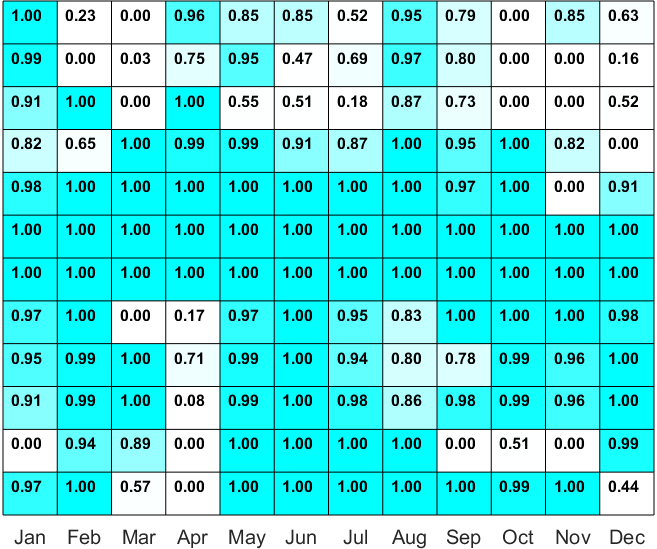}
    \captionof{figure}{R squared values of extrapolation performance of the linear regression model}
    \label{Fig6}
\end{minipage}
\noindent
\begin{minipage}{\linewidth}
\centering
\begin{minipage}{0.48\linewidth}
\begin{figure} [H]
\centering     \includegraphics[scale=0.325]{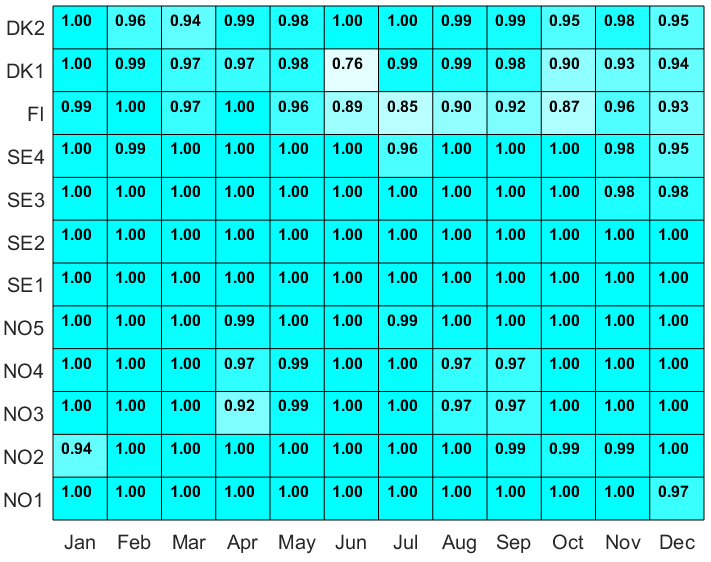}
    \caption{R squared values of interpolation performance of the polynomial regression model}
     \label{Fig7}
\end{figure}
\end{minipage}
\hspace{0.02\linewidth}
\begin{minipage}{0.48\linewidth}
\begin{figure} [H]
\vspace{-1mm}
\centering     \includegraphics[scale=0.325]{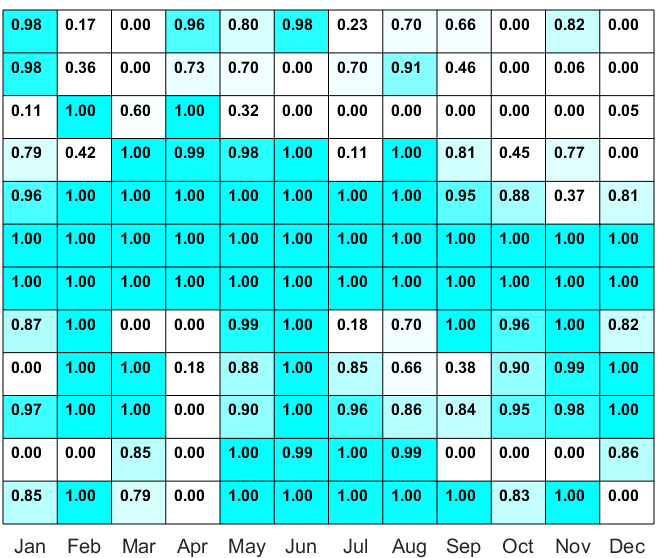}
    \caption{R squared values of extrapolation performance of the polynomial regression model}
     \label{Fig8}
\end{figure}
\end{minipage}
\end{minipage}

Tables \ref{tbl2} and \ref{tbl3} present the minimum, mean, and maximum values of the MSE (Mean Squared Errors) and SSE (Sum of Squared Errors) for the proposed TS-CRR algorithm, as well as for the linear, and polynomial regression models over all areas and months. 

\begin{table}[htbp] 
	\centering
	\caption{minimum, mean, and maximum values of MSE of the proposed TS-CRR algorithm, linear, and polynomial models over all areas and months divided by interpolation and extrapolation parts. 
	\label{tbl2}}
		\begin{tabular}{|c|c|c|c|}\hline
			model&Linear&Polynomial&\textbf{Proposed}\\\hline
			Min of interpolation MSE&0.0000&0.0000&0.0000 \\\hline
			Mean of interpolation MSE&0.0002&0.0001&0.0001 \\\hline
			Max of interpolation MSE&0.0106&0.0016&0.0015 \\\hline
   	  Min of extrapolation MSE&0.0000&0.0000&0.0000 \\\hline
			Mean of extrapolation MSE&0.0002&0.0005&0.0003 \\\hline
			Max of extrapolation MSE&0.0089&0.0138&0.0069 \\\hline
		\end{tabular}	
\end{table}
\begin{table}[htbp] \label{tbl3}
	\centering
	\caption{minimum, mean, and maximum values of SSE of the proposed TS-CRR algorithm, linear, and polynomial models over all areas and months divided by interpolation and extrapolation parts. 
	\label{tbl3}}
		\begin{tabular}{|c|c|c|c|}\hline
			model&Linear&Polynomial&\textbf{Proposed}\\\hline
			Min of interpolation SSE&0.0000&0.0000&0.0000 \\\hline
			Mean of interpolation SSE&0.0675&0.0183&0.0011 \\\hline
			Max of interpolation SSE&3.6394&0.5498&0.0258 \\\hline
   	  Min of extrapolation SSE&0.0000&0.0000&0.0000 \\\hline
			Mean of extrapolation SSE&0.0069&0.0168&0.0174 \\\hline
			Max of extrapolation SSE&0.3189&0.4969&0.5286 \\\hline
		\end{tabular}	
\end{table}

Regarding Table~\ref{tbl3}, it is important to note that the number of test data points used to evaluate interpolation performance is significantly larger than those used for extrapolation performance, as evident from Fig.~\ref{Fig3} and Fig.~\ref{Fig4}. Since SSE depends on the number of test data points, the larger test set used for interpolation results in higher SSE values compared to extrapolation.

Fig.~\ref{Fig5} and Fig.~\ref{Fig6} show that the linear regression model performs poorly in both interpolation and extrapolation tasks due to the limited number of monomials included in its structure. 

In contrast, as illustrated by Fig.~\ref{Fig7} and Table~\ref{tbl2}, the polynomial regression model demonstrates superior interpolation performance compared to linear regression model and comparable performance compared to the TS-CRR algorithm. However, Fig.~\ref{Fig8} and Table~\ref{tbl3} reveal that the polynomial regression model performs poorly in extrapolation (Increasing the number of monomials improves interpolation but tends to deteriorate extrapolation performance due to model overfitting).

Overall, the proposed TS-CRR algorithm shows satisfactory performance in both interpolation and extrapolation tasks when compared to conventional static regression models. The method not only achieves lower prediction errors in most tested scenarios, but also demonstrates greater robustness when applied to datasets with varying levels of noise and sparsity. These results suggest that the algorithm can effectively capture the underlying spatiotemporal relationships in the data, making it a promising alternative for applications where traditional regression approaches may struggle to generalize beyond the training domain.
\subsubsection{Comparison with intelligence models}
In this subsection, the results of the TS-CRR algorithm are compared with two intelligent models. First, the results of the binary decision tree and ensemble regression models are given in Fig. \ref{Fig9} to Fig. \ref{Fig12} \cite{myles2004introduction}.  

\noindent
\begin{minipage}{\linewidth} 
\centering
\begin{minipage}{0.48\linewidth}
\begin{figure} [H]
\centering     \includegraphics[scale=0.325]{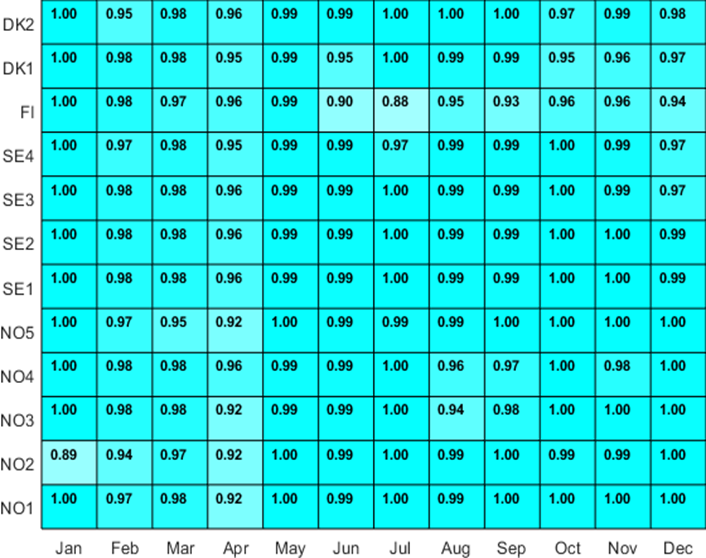}
    \caption{R squared values of interpolation of  binary decision tree model}
     \label{Fig9}
\end{figure}
\end{minipage}
\hspace{0.02\linewidth}
\begin{minipage}{0.48\linewidth}
\begin{figure} [H]
\centering     \includegraphics[scale=0.335]{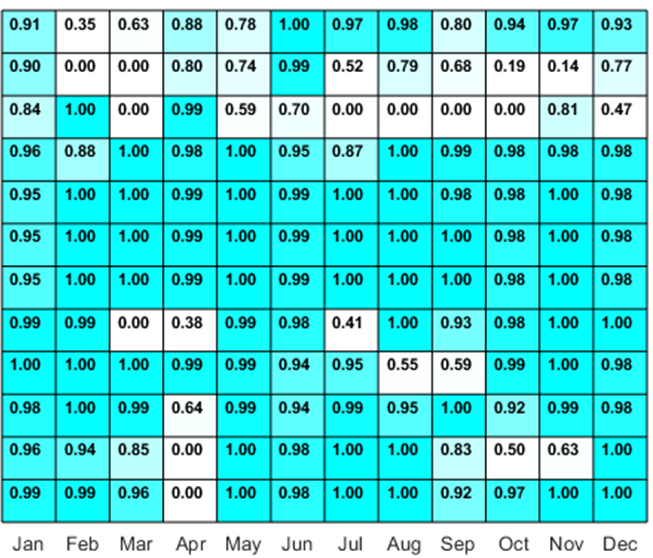}
    \caption{R squared values of performance of  binary decision tree model}
     \label{Fig10}
\end{figure} 
\end{minipage}
\end{minipage}
\noindent
\begin{minipage}{\linewidth}
\centering
\begin{minipage}{0.48\linewidth}
\begin{figure} [H]
\centering     \includegraphics[scale=0.325]{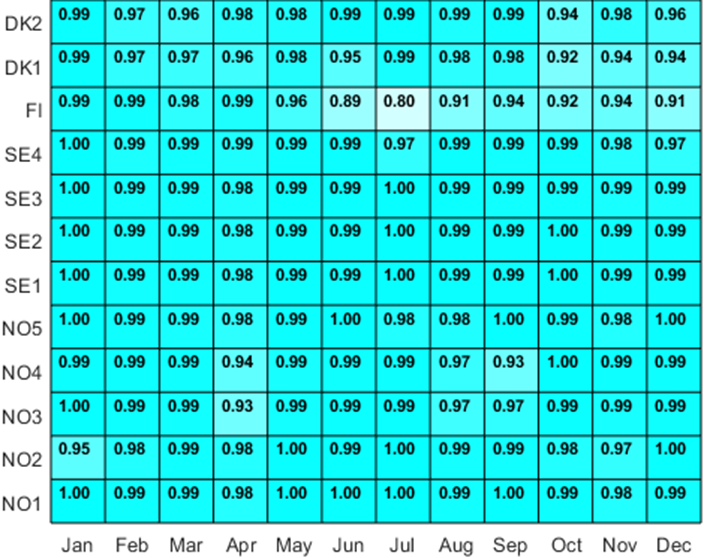}
    \caption{R squared values of interpolation of  ensemble regression model}
     \label{Fig11}
\end{figure}
\end{minipage}
\hspace{0.01\linewidth}
\begin{minipage}{0.48\linewidth}
\begin{figure} [H]
\centering     \includegraphics[scale=0.33]{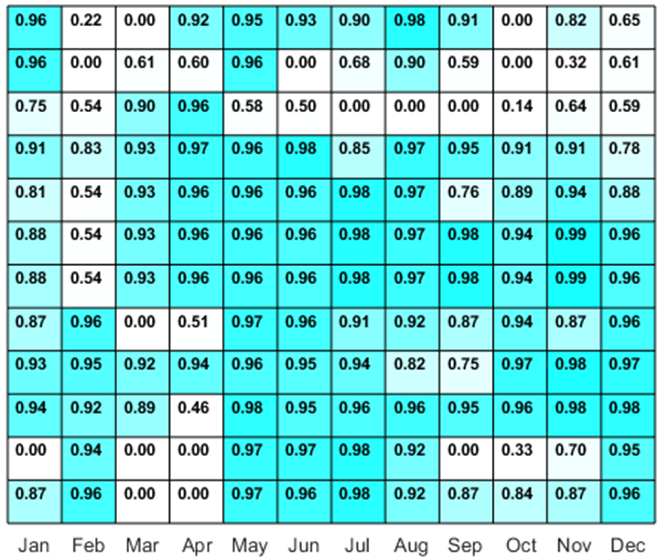}
    \caption{R squared values of extrapolation of  ensemble regression model}
     \label{Fig12}
\end{figure}
\end{minipage}
\end{minipage}

If we compare the R-squared values of the binary decision tree and ensemble regression models with those from the TS-CRR algorithm, we can clearly see the better performance of the latter in the examined cases.

For a better comparison, Fig. \ref{Fig13} shows the MSEs of the intelligent models and the TS-CRR algorithm using candle plots, where color intensity depicts the histogram of the MSE values. We have 144 interpolation MSEs and 144 extrapolation MSEs for each model, and their distributions are shown through candles in this figure.   
\begin{figure} [H]
\centering
     \includegraphics[scale=0.54]{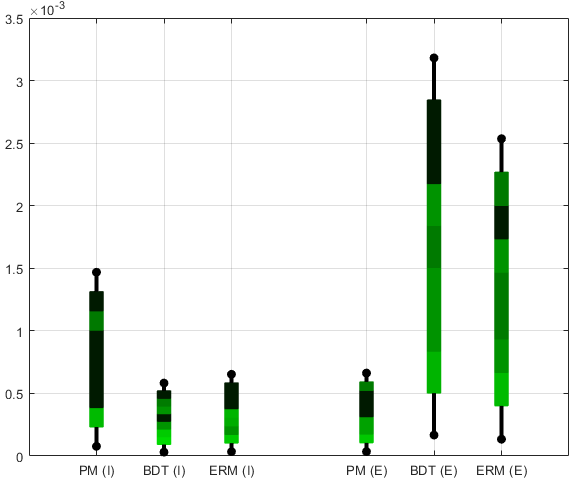}
    \caption{The histograms of interpolation and extrapolation MSEs for the TS-CRR algorithm (PM), Binary decision tree (BDT), and Ensemble Regression Model (ERM). Letters I and E are used to refer to interpolation and extrapolation MSEs, respectively.}
     \label{Fig13}
\end{figure}
Based on Fig.~\ref{Fig13}, the following items are needed to be discussed.
\begin{itemize}
\item Although the interpolation performance of the TS-CRR algorithm is good, the intelligent models exhibit even better interpolation capability. This is because intelligent models possess more complex architectures and richer parameterization, which enable them to better fit the training data, thereby achieving lower interpolation errors in most cases.
\item The extrapolation performance of the TS-CRR algorithm, however, is significantly better than that of the examined intelligent models. This is mainly due to the lower structural complexity of the TS-CRR algorithm, which reduces the risk of overfitting and makes it inherently more robust and suitable for extrapolation tasks, as observed here.
\item Overall, when considering both interpolation and extrapolation mean squared errors (MSEs) together, the TS-CRR delivers more balanced and superior performance compared to the intelligent models, making it a strong candidate for applications requiring accurate interpolation and extrapolation.
\end{itemize}

\label{sec:app1}
\subsection{Application to temperature measurements}
\noindent In order to evaluate the proposed TS-CRR algorithm on a different dataset, we employ temperature measurements at multiple international weather stations. This task is motivated by the dynamic and nonlinear dependencies between temperature measurements and key atmospheric variables, namely: pressure, relative humidity, and wind speed.

We focus on 12 weather stations listed in Table \ref{station_list}, located in different geographic regions and climates. Each station provides hourly measurements for 12 months (one full year). The goal is to forecast the temperature using the remaining three variables as input. The forecasting performance of the TS-CRR algorithm is compared against three benchmark models from two categories.
\begin{table}[h!]
\centering
\caption{List of weather stations.
\label{station_list}}
\begin{tabular}{|c|c|c|}
\hline
WMO ID & City & Country \\ \hline
3163 & Tiksi & Russia \\
8181 & Santo Domingo & Dominican Republic \\
12001 & Helgoland & Germany \\
12772 & Yerevan & Armenia \\
12805 & Tbilisi & Georgia \\
12812 & Kutaisi & Georgia \\
12836 & Gyumri & Armenia \\
12840 & Van & Turkey \\
12846 & Batumi & Georgia \\
12870 & Trabzon & Turkey \\
12882 & Kars & Turkey \\
12892 & Iğdır & Turkey \\ \hline
\end{tabular}
\end{table}
Interpolation and extrapolation errors are employed as performance indicators. For interpolation, odd-hour samples are used for training, and even-hour samples for testing. For extrapolation, the first 97\% of each month's data are used for training, and the final 3\% (roughly one day) for testing.
\subsection{Results of the Proposed TS-CRR algorithm}
\noindent The TS-CRR algorithm is applied to each of the 12 stations over the 12-month period. The model parameters are summarized in Table \ref{tab_temp_params}. R-squared values across all stations and months are shown in Figs. \ref{TempFig1} and \ref{TempFig2}.
\begin{table}[h!]
	\centering     
    \caption{Parameters of TS-CRR algorithm for temperature forecasting.
    \label{tab_temp_params}}
	\begin{tabular}{|c|l|c|}
		\hline
		\textbf{Parameter} & \textbf{Description} & \textbf{Value} \\\hline
		$d$     & Degree of the sparse polynomial     & 4    \\\hline
		$N$     & Number of training data points      & 360  \\\hline
		$m_d$   & Total number of candidate monomials & 35 \\\hline
		$N - l_b$ & Number of anomalous data points     & 6    \\\hline
		$l_m$   & Number of selected monomials  & 32 \\\hline
	\end{tabular}
\end{table}

\noindent
\begin{minipage}[h!]{0.48\linewidth}
    \centering
    \includegraphics[scale=0.33]{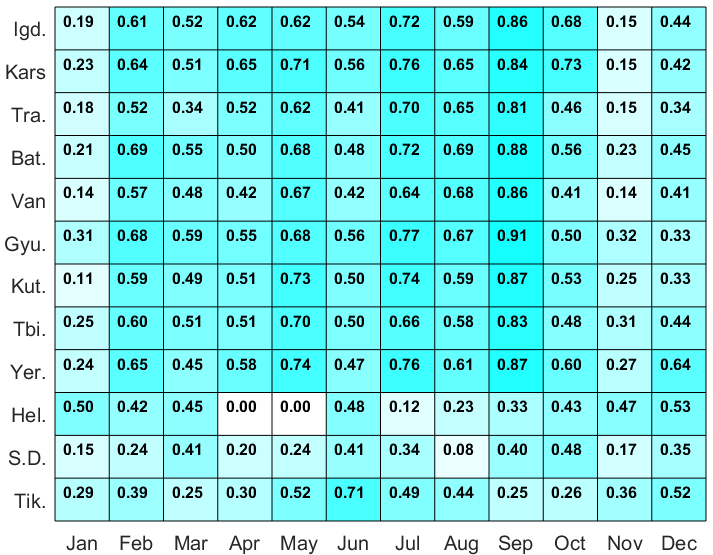}
    \captionof{figure}{R-squared values of interpolation performance of the TS-CRR}
    \label{TempFig1}
\end{minipage}
\hfill
\begin{minipage}[h!]{0.48\linewidth}
    \centering
    \includegraphics[scale=0.33]{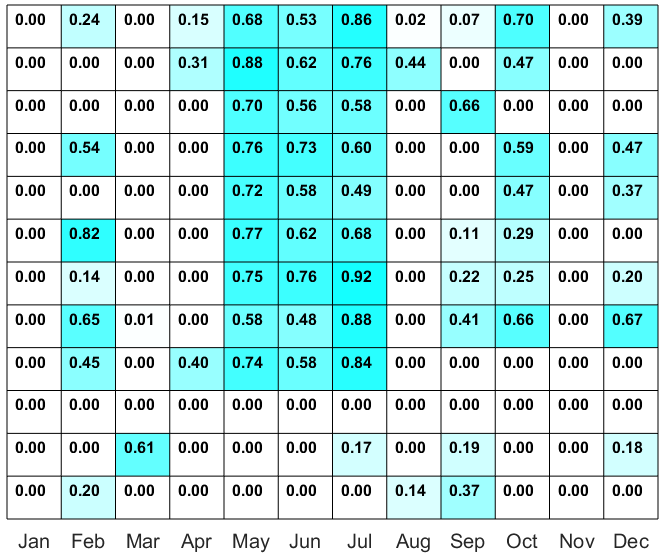}
    \captionof{figure}{R-squared values of extrapolation performance of the TS-CRR }
    \label{TempFig2}
\end{minipage}
\noindent
\begin{minipage}{\linewidth}
\centering
\begin{minipage}{0.48\linewidth}
\begin{figure} [H] 
\centering     \includegraphics[scale=0.325]{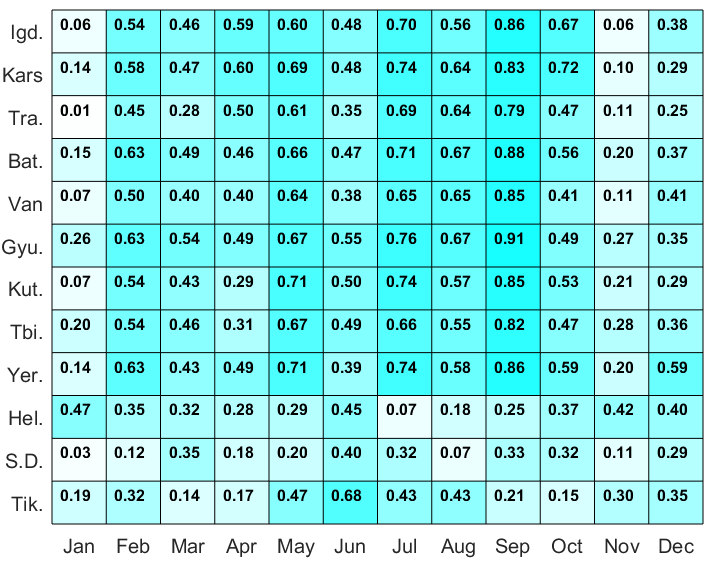}
    \caption{R-squared values of interpolation of the linear regression}
     \label{TempFig3}
\end{figure}
\end{minipage}
\hspace{0.01\linewidth}
\begin{minipage}{0.48\linewidth}
\vspace{0.1mm}
\begin{figure} [H] 
\centering     \includegraphics[scale=0.325]{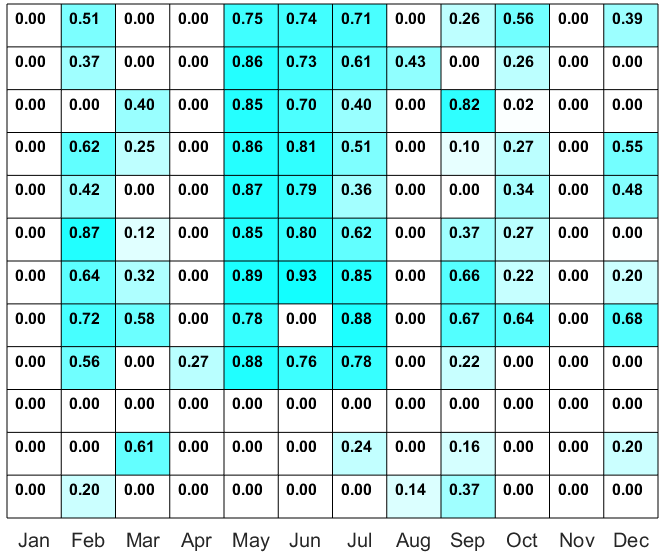}
    \caption{R-squared values of extrapolation of the linear regression}
     \label{TempFig4}
\end{figure}
\end{minipage}
\end{minipage}
\noindent
\begin{minipage}{\linewidth}
\centering
\begin{minipage}{0.48\linewidth}
\begin{figure} [H] 
\centering     \includegraphics[scale=0.325]{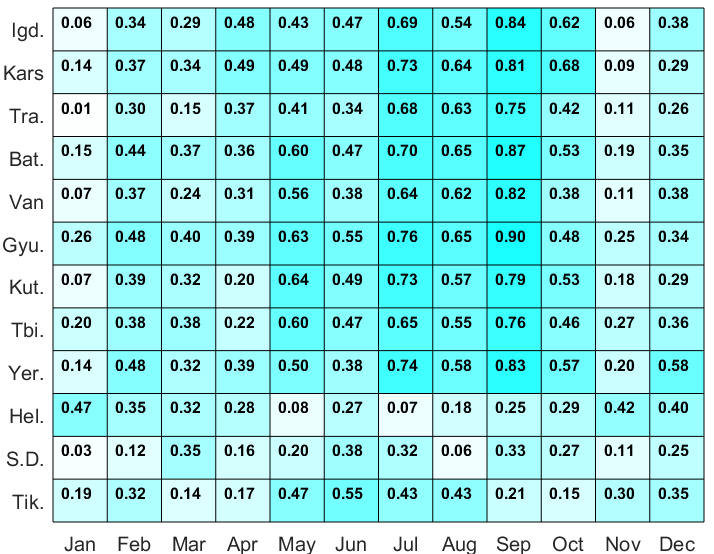}
    \caption{R-squared values of interpolation of the polynomial regression}
     \label{TempFig5}
\end{figure}
\end{minipage}
\hspace{0.01\linewidth}
\begin{minipage}{0.48\linewidth}
\begin{figure} [H] 
\centering     \includegraphics[scale=0.325]{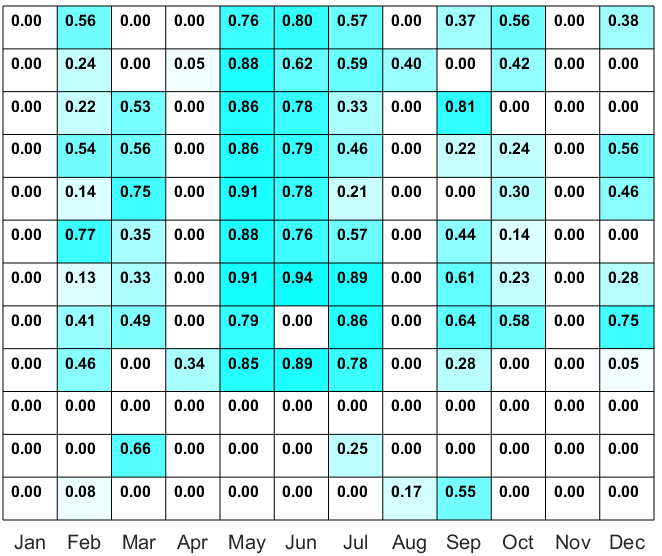}
    \caption{R-squared values of extrapolation of the polynomial regression}
     \label{TempFig6}
\end{figure}
\end{minipage}
\end{minipage}

\subsubsection{Comparison with static regression and intelligent models}
The results of linear and polynomial regression models are shown by Fig. \ref{TempFig3} to Fig. \ref{TempFig6}. Then, Fig. \ref{TempFig7} shows the histogram candles for MSE distributions of interpolation and extrapolation in the proposed, linear, and polynomial regressions.

\begin{figure} [h!]
\centering
     \includegraphics[scale=0.35]{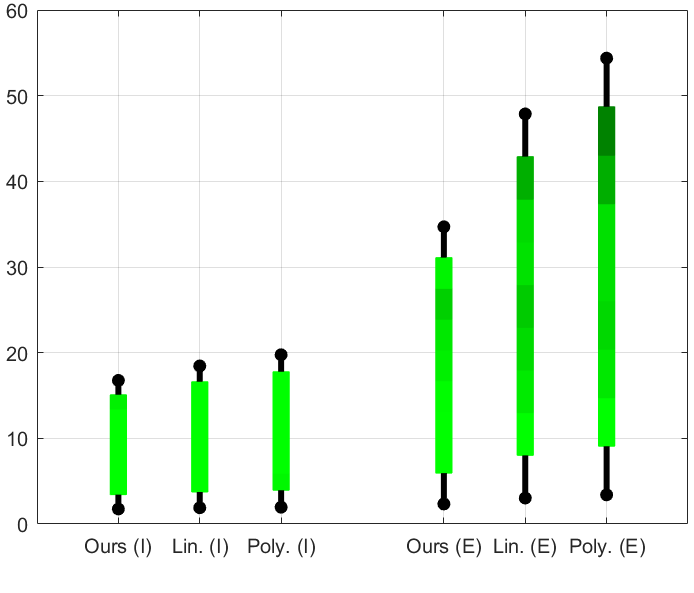}
    \caption{The histograms of interpolation and extrapolation MSEs in the proposed (Ours), linear (Lin), and polynomial (Poly) regression models. Letters I and E are used to refer to interpolation and extrapolation MSEs, respectively.}
     \label{TempFig7}
\end{figure}
\vspace{-8mm}
\noindent
\begin{minipage}{\linewidth}
\centering
\begin{minipage}{0.48\linewidth}
\begin{figure} [H] 
\centering
 \includegraphics[scale=0.325]{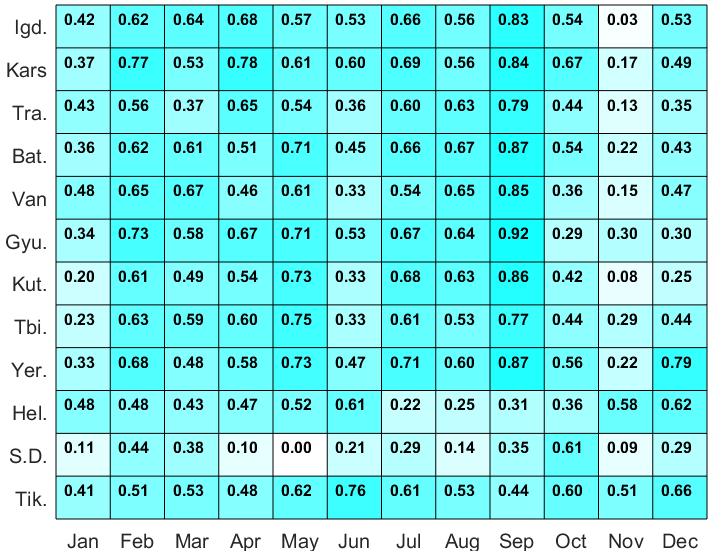}
    \caption{R-squared values of interpolation of the binary decision tree regression}
     \label{TempFig8}
\end{figure}
\end{minipage}
\hspace{0.02\linewidth}
\begin{minipage}{0.48\linewidth}
\vspace{0.2mm}
\begin{figure} [H] 
\centering    \includegraphics[scale=0.325]{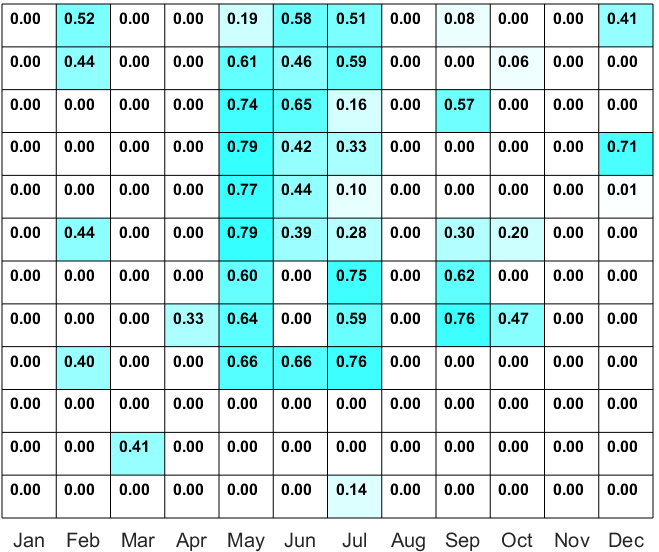}
    \caption{R-squared values of extrapolation of the binary decision tree regression}
     \label{TempFig9}
\end{figure}
\end{minipage}
\end{minipage}
Finally, the TS-CRR algorithm is also benchmarked against the binary decision tree. The R-squared values for the binary decision tree are shown in Fig. \ref{TempFig8} and Fig. \ref{TempFig9}.

Fig.~\ref{TempFig8} and Fig.~\ref{TempFig9} show the interpolation and extrapolation performances of the binary decision tree regression model across the 12 weather stations. The results indicate that while the binary decision tree achieves moderate interpolation accuracy, its extrapolation performance is significantly weaker. This behavior highlights the model’s tendency to overfit the training data, capturing local variations but failing to generalize well to unseen samples. In contrast, the proposed TS-CRR algorithm demonstrates more stable performance across both interpolation and extrapolation.
\subsection{Discussion}
The proposed TS-CRR algorithm demonstrates promising performance in both electricity price and temperature measurement applications, outperforming several benchmark regression and intelligent models in various metrics.
It avoids overfitting through sparsity enforcement while preserving enough complexity to capture essential nonlinearities. This results in strong extrapolation performance without substantial compromise in interpolation accuracy. Its simplicity, interpretability, and robustness make it suitable for deployment in real-world forecasting systems where both short-term accuracy and long-term generalization are critical.

\section{Conclusions} \label{sec:cnc}
This paper proposes a new SPRM with anomalous data filtering. First, the SPRM is formulated as a nonconvex QCQP. Then through a proposed fractional mapping, the nonconvex QCQP is reformulated as a FP with only one nonconvex spherical constraint. We theoretically show that the reformulated FP has better computational performance than the original nonconvex QCQP. The proposed FP is then solved by a developed linear-based relaxation model. A TS-CRR algorithm is proposed for SPRM with anomalous data filtering. The TS-CRR algorithm, in step 1, finds optimal monomials and then in step 2 optimally estimates the corresponding coefficients. We have tested the proposed TS-CRR algorithm using two datasets: one contains electricity prices and the other consists of temperature measurements. The results of the TS-CRR algorithm are compared with those obtained from four different different benchmark models. The numerical experiments supported by theoretical developments show a promising performance of the proposed TS-CRR algorithm for real-life applications. 

\section{Data availability statement}
All datasets used in the current paper are available through the following link: 

\url{https://github.com/Roozbeh-Abolpour/Paper-Datasets}.


\end{document}